\documentclass[11pt]{article}
\usepackage{amsfonts, amsmath, amsthm, amssymb, mathtools, enumitem, tikz, hyperref, caption}
\usepackage{eucal}
\usepackage[noabbrev,capitalize]{cleveref}

\definecolor{darkgray}{RGB}{64,64,64}
\definecolor{litegray}{RGB}{192,192,192}
\definecolor{green}{HTML}{0F9D58}
\definecolor{red}{HTML}{DB4437}
\tikzstyle{vertex}=[circle, draw, fill=red, inner sep=0pt, minimum width=5pt]
\tikzstyle{vtx}=[circle, draw, fill=litegray, inner sep=0pt, minimum width=5pt]

\usepackage{microtype, fullpage}
\newtheorem{theorem}{Theorem}[section]

\newtheorem{proposition}[theorem]{Proposition}
\newtheorem{corollary}[theorem]{Corollary}
\newtheorem{conjecture}[theorem]{Conjecture}

\theoremstyle{definition}
\newtheorem{definition}[theorem]{Definition}

\theoremstyle{remark}

\newtheorem*{claim*}{Claim}

\newcommand{\ns}{\mathrm{ns}}
\newcommand{\ff}{\mathrm{ff}}
\newcommand{\dff}{\textrm{-}\mathrm{ff}}
\newcommand{\FF}{\mathcal{F}}
\newcommand{\GG}{\mathcal{G}}
\newcommand{\HH}{\mathcal{H}}

\newcommand{\F}{\mathbb{F}}
\newcommand{\E}{\mathbb{E}}

\title{Near-sunflowers and focal families}

\author{
  Noga Alon\footnote{Department of Mathematics, Princeton University, Princeton, NJ 08544, USA and Schools of Mathematics and Computer Science, Tel Aviv University, Tel Aviv 69978, Israel. Email: {\tt nogaa@tau.ac.il}. Research supported in part by NSF grant DMS-1855464, ISF grant 281/17, BSF grant 2018267 and the Simons Foundation.}
  \and
  Ron Holzman\footnote{Department of Mathematics, Technion -- Israel Institute of Technology, Technion City, Haifa 3200003, Israel. Email: {\tt holzman@technion.ac.il}. Research done during a visit at the Department of Mathematics, Princeton University, supported by the H2020-MSCA-RISE project CoSP--GA No. 823748.}
}

\date{}

\begin{document}

\maketitle

\begin{abstract}
We present some problems and results about variants of sunflowers in
families of sets. In particular, we improve an upper bound of
the first author,
K\"orner and Monti on the maximum number of binary vectors
of length $n$ so that every four of them are split into two pairs
by some coordinate. We also propose a weaker version of
the Erd\H{o}s-Rado sunflower conjecture.
\end{abstract}

\section{Introduction} \label{sec:introduction}

Introduced by Erd\H{o}s and Rado~\cite{ER}, sunflowers (also called strong $\Delta$-systems) have a long history of study and applications in extremal combinatorics and theoretical computer science. Recall that a family $\HH$ of $r$ distinct subsets of $[n] = \{1, 2,\ldots, n\}$ is called a \emph{sunflower} of size $r$ if every $i \in [n]$ belongs to either $0$, $1$ or $r$ of the sets in $\HH$.

Erd\H{o}s and Rado famously conjectured that if $\FF$ is a $k$-uniform family of sets (i.e., $|A| = k$ for every $A \in \FF$) not containing a sunflower of size $r$, then $|\FF| \le C^k$, where $C$ is a constant depending only on $r$.
For many years, the best known upper bound was close to $k!$
for any fixed $r$. A recent breakthrough due
to Alweiss, Lovett, Wu and Zhang~\cite{ALWZ}
improved the bound to $(\log k)^{(1+o(1))k}$ for any fixed $r$,
but the original conjecture is still open even for $r=3$.

Seeking a bound that depends on the size $n$ of the ground set rather than the uniformity $k$, Erd\H{o}s and Szemer\'edi~\cite{ES} conjectured that if $\FF$ is a family of subsets of $[n]$ not containing a sunflower of size $r$, then $|\FF| \le c^n$, where $c < 2$ is a constant depending only on $r$. They showed (implicitly, made explicit by Deuber et al.~\cite{DEGKM}) that their conjecture would follow from the Erd\H{o}s-Rado conjecture. The recent solution by Ellenberg and Gijswijt~\cite{EG} of the cap set problem, confirmed the $r=3$ case of the Erd\H{o}s-Szemer\'edi conjecture (via a reduction due to
the first author, Shpilka and Umans~\cite{ASU}); see also Naslund and Sawin~\cite{NS} and Heged\H{u}s~\cite{H} for explicit bounds. But the conjecture is still open for $r \ge 4$.

In this paper we introduce a weaker variant of sunflowers.

\begin{definition} \label{def:ns}
A family $\HH$ of $r$ distinct subsets of $[n]$ is called a \emph{near-sunflower} of size $r$ if every $i \in [n]$ belongs to either $0$, $1$, $r-1$ or $r$ of the sets in $\HH$.
\end{definition}

The weakening consists in adding the option of belonging to $r-1$ of the sets (this renders the property interesting only for $r \ge 4$). It is natural in that it makes the property symmetric: if $\HH$ is a near-sunflower then so is $\{[n] \setminus A:\, A \in \HH\}$. One may hope that when sunflowers are replaced by near-sunflowers, the notoriously difficult conjectures of Erd\H{o}s-Rado and Erd\H{o}s-Szemer\'edi will become easier. We show that this is indeed the case in the Erd\H{o}s-Szemer\'edi setting (bound depending on $n$), but leave the question in the Erd\H{o}s-Rado setting (bound depending on $k$) open.

The following variant of near-sunflowers, in which one member of the family plays a distinguished role, will be of interest. It is convenient to define it for binary vectors of length $n$ instead of subsets of $[n]$ -- henceforth we will pass freely between these two equivalent formalisms.

\begin{definition} \label{def:ff}
A family $x^{(0)}, x^{(1)},\ldots, x^{(r-1)}$ of $r$ distinct vectors in $\{0,1\}^n$ is \emph{focal} with focus $x^{(0)}$ if for every coordinate $i \in [n]$ at least $r-2$ of the $r-1$ entries $x^{(1)}_i,\ldots, x^{(r-1)}_i$ are equal to $x^{(0)}_i$.
\end{definition}

Thus, a focal family is a near-sunflower with the additional property that one of the vectors -- the focus -- is always in the majority. Unlike near-sunflowers, focal families are interesting already for $r=3$. While sunflowers and focal families are both special kinds of near-sunflowers, they are not logically comparable to each other.

The two extremal functions corresponding to our definitions are:
\begin{eqnarray*}
g^{\ns}_r(n) & = & \max \{|\FF|:\, \FF \subseteq \{0,1\}^n \textrm{ contains no near-sunflower of size }r\}\\
g^{\ff}_r(n) & = & \max \{|\FF|:\, \FF \subseteq \{0,1\}^n \textrm{ contains no focal family of size }r\}
\end{eqnarray*}

It follows from the definitions that $g^{\ns}_r(n) \le g^{\ff}_r(n)$. Our main result gives upper and lower bounds for these functions.

\begin{theorem} \label{thm:main}
For $r \ge 3$ we have:
\begin{itemize}
\item[(a)] $g^{\mathrm{ns}}_r(n) \le g^{\mathrm{ff}}_r(n) \le (r-1) 2^{\lceil \frac{(r-2)n}{r-1} \rceil}$.
\item[(b)] There exist positive constants $c^{\mathrm{ns}}_r$ and $c^{\mathrm{ff}}_r$ so that
\begin{eqnarray*}
g^{\mathrm{ns}}_r(n) & \ge & c^{\mathrm{ns}}_r (\frac{2}{(r+1)^{\frac{1}{r-1}}})^n,\\
g^{\mathrm{ff}}_r(n) & \ge & c^{\mathrm{ff}}_r (\frac{2}{r^{\frac{1}{r-1}}})^n.
\end{eqnarray*}
\end{itemize}
\end{theorem}

In particular, for $r=4$, our bounds (ignoring constants) are $2^{\frac{2n}{3}}$ from above and $(\frac{8}{5})^{\frac{n}{3}}$ and $2^{\frac{n}{3}}$ from below, for near-sunflowers and focal families, respectively. Families without near-sunflowers of size $4$ were previously studied (with different terminology) by
the first author, K\"orner and Monti~\cite{AKM}, settling a problem
suggested by S\'os  in the late 80's.
While their lower bound was the same as ours, their upper bound was roughly $2^{0.773n}$, with a proof based on Sauer's lemma. It is remarkable that our short and elementary proof improves their bound. We note that they also extended their result to $r > 4$, but with a different definition. Whereas our near-sunflowers allow an element to belong to $0$, $1$, $r-1$ or $r$ of the sets, their definition allows everything except one forbidden value.

We also remark that a concept analogous to our focal families, where instead of requiring ``at least $r-2$" in Definition~\ref{def:ff} one requires ``at least $1$," was studied in coding theory under the names separating codes (Cohen and Schaathun~\cite{CS}) and frameproof codes
(Blackburn~\cite{B}). This case is somewhat simpler and the
bounds obtained in those papers coincide with ours for
the special case $r=3$.

More generally, our paper follows a long line of literature in extremal combinatorics, information theory and coding theory. The common thread is bounding the largest possible cardinality of a family of vectors of length $n$, so that for any $r$ of them there exist coordinates displaying certain desirable patterns. For $r=2$, Sperner's~\cite{S} classical theorem on antichains is a prime example. For $r=3$ we mention the theorem of Erd\H{o}s, Frankl and F\"uredi~\cite{EFF} on families in which no set is covered by the union of two others; the problem of cancellative families solved by Tolhuizen~\cite{T}; and a variety of related problems described by K\"orner~\cite{K}. For $r=4$ there is Lindstr\"om's~\cite{L} theorem on determining two vectors from their sum modulo $2$, and K\"orner and Simonyi's~\cite{KS} bounds for two-different quadruples. For general $r$, we refer to the study of disjunctive codes (Dyachkov and Rykov~\cite{DR}). In all of these problems, and many others, the cardinality of the largest family grows exponentially in $n$, but (with few exceptions) the asymptotic growth rate is not known. Our problems are no exception.

The proof of Theorem~\ref{thm:main} is given in the next section. In Section~\ref{sec:q} we adapt the definition and the bounds for focal families to vectors over larger alphabets, noting that the upper bound becomes essentially tight when the size of the alphabet exceeds $n$. We show in Section~\ref{sec:linear} that the upper bound in Theorem~\ref{thm:main} can be improved if the family of vectors is closed under addition modulo $2$ (i.e., forms a linear code). In Section~\ref{sec:one-sided} we consider one-sided focal families, where $0$ and $1$ entries are treated asymmetrically, and obtain corresponding bounds. Finally, in Section~\ref{sec:k} we discuss the challenge of obtaining an exponential upper bound in terms of the uniformity $k$, and prove such a bound under a stronger condition.

\section{Proof of Theorem~\ref{thm:main}} \label{sec:proof}
\paragraph{The upper bound.}
Let $\FF \subseteq \{0,1\}^n$ have cardinality $|\FF| > (r-1) 2^{\lceil \frac{(r-2)n}{r-1} \rceil}$. We have to show that $\FF$ contains a focal family of size $r$. Fix a partition $A_1,\ldots, A_{r-1}$ of $[n]$ into $r-1$ parts of size $|A_j| \ge \lfloor \frac{n}{r-1} \rfloor$ each. For a subset $S \subseteq [r-1]$ with $|S|=r-2$, say that a vector $x \in \mathcal{F}$ is $S$-\emph{unique} if there is no other vector in $\mathcal{F}$ with the same projection on $\bigcup_{j \in S} A_j$. Since $|\bigcup_{j \in S} A_j| \le \lceil \frac{(r-2)n}{r-1} \rceil$, for a given $S$ the number of $S$-unique vectors in $\mathcal{F}$ is at most $2^{\lceil \frac{(r-2)n}{r-1} \rceil}$. It follows from our assumption on $|\mathcal{F}|$ that there exists a vector $x^{(0)} \in \mathcal{F}$ which is not $S$-unique for any $S \subseteq [r-1]$ with $|S|=r-2$. This means that we can find vectors $x^{(1)},\ldots, x^{(r-1)} \in \mathcal{F} \setminus \{x^{(0)}\}$ so that each $x^{(j)}$ agrees with $x^{(0)}$ except possibly on coordinates in $A_j$. Note that $x^{(1)},\ldots, x^{(r-1)}$ are pairwise distinct, because if two of them were equal they would have to coincide with $x^{(0)}$. By construction, the subfamily $x^{(0)}, x^{(1)},\ldots, x^{(r-1)}$ is focal with focus $x^{(0)}$.

\paragraph{The lower bounds.}
As is common in such problems, we use random choice with alterations. We describe the argument for near-sunflowers, later pointing out how to adapt it to focal families.

We start by forming a random family $\GG \subseteq \{0,1\}^n$ to which each vector $x \in \{0,1\}^n$ belongs, independently, with probability $p$ (to be determined later). Then $\E(|\GG|) = 2^n p$. Let $N_{\GG}$ be a random variable counting the number of near-sunflowers of size $r$ contained in $\GG$. By removing at most $N_{\GG}$ vectors from $\GG$, we obtain a family $\FF$ of cardinality at least $|\GG| - N_{\GG}$ which contains no near-sunflower of size $r$. By linearity of expectation, $\E(|\FF|) \ge 2^n p - N^{\ns}_r p^r$, where $N^{\ns}_r$ is the number of near-sunflowers of size $r$ in $\{0,1\}^n$.

To estimate $N^{\ns}_r$, note that the number of $r \times n$ binary matrices so that the number of $1$ entries in each column is $0$, $1$, $r-1$ or $r$ is $(2r+2)^n$. Since near-sunflowers correspond to such matrices with distinct rows, and the order of the rows is immaterial, it follows that $N^{\ns}_r \le \frac{1}{r!} (2r+2)^n$. Thus,
\[ \E(|\FF|) \ge 2^n p - \frac{1}{r!} (2r+2)^n p^r,\]
and choosing $p = c(\frac{1}{(r+1)^{\frac{1}{r-1}}})^n$ for a suitable $c=c(r) > 0$ yields $\E(|\FF|) \ge c^{\ns}_r (\frac{2}{(r+1)^{\frac{1}{r-1}}})^n$ for some $c^{\ns}_r > 0$.
Hence, there is a realization of $\FF$ having at least this cardinality.

Moving to focal families, the argument is similar, but now we have to estimate the number $N^{\ff}_r$ of focal families of size $r$ in $\{0,1\}^n$. The number of $r \times n$ binary matrices so that in each column the first entry is repeated at least $r-2$ times among the other entries is $(2r)^n$. Since focal families correspond to such matrices with distinct rows, and the order of the last $r-1$ rows is immaterial, it follows that $N^{\ff}_r \le \frac{1}{(r-1)!} (2r)^n$. Thus,
\[ \E(|\FF|) \ge 2^n p - \frac{1}{(r-1)!} (2r)^n p^r,\]
and choosing $p = c(\frac{1}{r^{\frac{1}{r-1}}})^n$ for a suitable
$c=c(r) > 0$ yields $\E(|\FF|) \ge c^{\ff}_r (\frac{2}{r^{\frac{1}{r-1}}})^n$ for some $c^{\ff}_r > 0$, as required. \hspace{\stretch{1}}$\square$

\section{Focal families over larger alphabets} \label{sec:q}
Given any integer $q \ge 2$, Definition~\ref{def:ff} can be applied verbatim to vectors in $[q]^n$ to define $q$-ary focal families. Let $g^{q\dff}_r(n)$ be the corresponding extremal function. A straightforward adaptation of the proof above yields the following version of Theorem~\ref{thm:main} for $q$-ary focal families.

\begin{theorem} \label{thm:q}
For $q \ge 2$ and $r \ge 3$ we have
\[ c^{q\dff}_r (\frac{q}{((q-1)(r-1)+1)^{\frac{1}{r-1}}})^n \le g^{q\dff}_r(n) \le (r-1) q^{\lceil \frac{(r-2)n}{r-1} \rceil} \]
for some positive constant $c^{q\dff}_r$.
\end{theorem}

When $q \ge n$ and $q$ is a prime power, we can replace the probabilistic lower bound by a constructive one which
matches (up to a constant factor depending on $r$) the upper bound.

\begin{proposition} \label{prop:q}
If $q \ge n$ and $q$ is a prime power then
\[ g^{q\dff}_r(n) \ge q^{\lceil \frac{(r-2)n}{r-1} \rceil}. \]
\end{proposition}

\begin{proof}
The Reed-Solomon code with suitable parameters gives the desired lower bound. For completeness, we describe the construction. We identify the elements of the finite field $\F_q$ with the $q$ symbols in our alphabet. We choose and fix $n$ distinct elements $a_1, a_2,\ldots, a_n \in \F_q$, and identify them with the coordinates $1, 2,\ldots, n$. There are $q^{\lceil \frac{(r-2)n}{r-1} \rceil}$ polynomials $p(x)$ of degree less than $\lceil \frac{(r-2)n}{r-1} \rceil$ over $\F_q$. With every such polynomial we associate the vector $(p(a_1), p(a_2),\ldots, p(a_n))$, which gives a family $\FF$ of $q$-ary vectors of length $n$, with $|\FF| = q^{\lceil \frac{(r-2)n}{r-1} \rceil}$.

We claim that $\FF$ contains no focal family of size $r$. Indeed, suppose that $x^{(0)}, x^{(1)}, \ldots, x^{(r-1)} \in \FF$ form such a family with focus $x^{(0)}$. Then by the pigeonhole principle, some $x^{(j)}$, $j \in [r-1]$, has to agree with $x^{(0)}$ on at least $\lceil \frac{(r-2)n}{r-1} \rceil$ coordinates. This means that the corresponding polynomials $p^{(j)}$ and $p^{(0)}$ agree on at least $\lceil \frac{(r-2)n}{r-1} \rceil$ elements of $\F_q$. But this is impossible, as they are distinct polynomials of degree less than $\lceil \frac{(r-2)n}{r-1} \rceil$.
\end{proof}

\section{Improved upper bound in the linear case} \label{sec:linear}
While Proposition~\ref{prop:q} shows that our upper bound is essentially tight when $q \ge n$, we believe that for $q=2$ and large $n$ it is not. To support this belief, we show here that the upper bound can be significantly improved if we restrict attention to families of binary vectors which are linear codes (i.e., closed under addition modulo $2$). This can be done for any value of $r$, but for simplicity and concreteness of the bound we do it for $r=4$.

We are going to use a known bound on the tradeoff between cardinality and minimum Hamming distance in a family $\FF$ of binary vectors of length $n$. Recall that, by the linear programming bound (McEliece, Rodemich, Rumsey and Welch~\cite{MRRW}), if the Hamming distance between any two distinct vectors in $\FF$ is greater than $\delta n$, then
\[ |\FF| \le 2^{(h(\frac{1}{2} - \sqrt{\delta(1 - \delta)}) + o(1))n}, \]
where $h(x)$ is the binary entropy function defined by
\[ h(x) = -x \log_2 x -(1-x) \log _2 (1-x). \]
We first prove the following theorem, which does not require linearity.

\begin{theorem} \label{thm:3diff}
Let $\mathcal{F}$ be a family of at least $2^{0.44n}$ subsets of $[n]$, where $n$ is large enough. Then there exist three pairs of distinct sets in $\mathcal{F}$ such that their symmetric differences $A \bigtriangleup B$, $C \bigtriangleup D$ and $E \bigtriangleup F$ are pairwise disjoint.
\end{theorem}

\begin{proof}
We apply the above-mentioned bound repeatedly. First, a calculation shows that for $\delta = 0.213$, we have $h(\frac{1}{2} - \sqrt{\delta(1 - \delta)}) < 0.44$. As $|\FF| \ge 2^{0.44n}$ and $n$ is large, the bound implies the existence of distinct sets $A, B \in \FF$ with $|A \bigtriangleup B| \le 0.213n$. Next, by the pigeonhole principle, we can find at least $\frac{|\mathcal{F}|}{2^{|A \bigtriangleup B|}}$ sets in $\mathcal{F}$ having the same intersection with $A \bigtriangleup B$. Let $\FF'$ be the family obtained by restricting these sets to $[n] \setminus (A \bigtriangleup B)$. A calculation shows that for $\delta' = 0.287$, we have $h(\frac{1}{2} - \sqrt{\delta'(1 - \delta')}) < 0.28$. As $|\FF'| \ge 2^{0.44n - |A \bigtriangleup B|} > 2^{0.28(n - |A \bigtriangleup B|)}$ and $n$ is large, the bound implies the existence of distinct sets $C, D \in \FF$ such that $(C \bigtriangleup D) \cap (A \bigtriangleup B) = \emptyset$ and $|C \bigtriangleup D| \le 0.287(n - |A \bigtriangleup B|)$. Now $|A \bigtriangleup B| + |C \bigtriangleup D| < 0.44n$, and again by the pigeonhole principle we can find two distinct sets $E, F \in \mathcal{F}$ having the same intersection with $(A \bigtriangleup B) \cup (C \bigtriangleup D)$, which completes the proof.
\end{proof}

\begin{corollary} \label{cor:linear}
Let $\FF$ be a linear subspace of $\{0,1\}^n$ of dimension at least $0.44n$, where $n$ is large enough. Then $\FF$ contains a focal family of size $4$.
\end{corollary}

\begin{proof}
Viewing $\FF$ as a family of subsets of $[n]$, it is closed under symmetric difference. Hence the theorem yields three pairwise disjoint non-empty sets $X^{(1)}, X^{(2)}, X^{(3)} \in \FF$. Taking the empty set as the focus $X^{(0)}$, we obtain a focal family of size $4$.
\end{proof}

\section{One-sided focal families} \label{sec:one-sided}
The requirement defining a focal family may be separated into two one-sided requirements as follows.

\begin{definition} \label{def:one-sided}
Let $b \in \{0,1\}$. A family $x^{(0)}, x^{(1)},\ldots, x^{(r-1)}$ of $r$ distinct vectors in $\{0,1\}^n$ is $b$-\emph{focal} with focus $x^{(0)}$ if for every coordinate $i \in [n]$ such that $x^{(0)}_i = b$, at least $r-2$ of the $r-1$ entries $x^{(1)}_i,\ldots, x^{(r-1)}_i$ are equal to $b$.
\end{definition}

The corresponding extremal functions for $b = 0,1$ are:
\[ g^{b\dff}_r(n) = \max \{|\FF|:\, \FF \subseteq \{0,1\}^n \textrm{ contains no $b$-focal family of size }r\} \]

It will be convenient to study the extremal questions first for $k$-uniform families. Let ${[n] \choose k}$ be the family of all $k$-element subsets of $[n]$. For $b = 0,1$ let:
\[ g^{b\dff}_r(n,k) = \max \{|\FF|:\, \FF \subseteq {[n] \choose k} \textrm{ contains no $b$-focal family of size }r\} \]

Since $\HH$ is $0$-focal if and only if $\{[n] \setminus A:\, A \in \HH\}$ is $1$-focal, we have $g^{0\dff}_r(n) = g^{1\dff}_r(n)$ and $g^{0\dff}_r(n,k) = g^{1\dff}_r(n,n-k)$. So we only need to study these questions for one value of $b$.

\begin{theorem} \label{thm:1-sided}
For $r \ge 3$ and $0 \le k \le n$ we have
\[ g^{1\dff}_r(n,k) \le (r-1) \frac{{n \choose \lceil \frac{(r-2)k}{r-1} \rceil}}{{k \choose \lceil \frac{(r-2)k}{r-1} \rceil}}. \]
\end{theorem}

\begin{proof}
Let $\FF$ be a family of $k$-element subsets of $[n]$ containing no $1$-focal family of size $r$. For a set $A \in \FF$, we say that a set $S$ is an \emph{own-subset} of $A$ if $S \subseteq A$ and $S \nsubseteq B$ for any $B \in \FF \setminus \{A\}$.

Consider an arbitrary $(r-1)$-tuple $A_1,\ldots, A_{r-1}$ of pairwise disjoint $\lfloor \frac{k}{r-1} \rfloor$-element subsets of a set $A \in \FF$. If for every $j \in [r-1]$ there exists a set $B_j \in \FF \setminus \{A\}$ such that $A \setminus A_j \subseteq B_j$, then the sets $A, B_1,\ldots, B_{r-1}$ form a $1$-focal family of size $r$ with focus $A$, contradicting our assumption on the family $\FF$. Hence there exists $j \in [r-1]$ so that $A \setminus A_j$ is an own-subset of $A$.

We claim that for a fixed set $A \in \FF$, the probability that a uniformly random $\lceil \frac{(r-2)k}{r-1} \rceil$-element subset $S$ of $A$ is an own-subset is at least $\frac{1}{r-1}$. Indeed, consider the following two-step random process. First, choose uniformly at random an $(r-1)$-tuple $A_1,\ldots, A_{r-1}$ of pairwise disjoint $\lfloor \frac{k}{r-1} \rfloor$-element subsets of $A$. Second, choose uniformly at random a value $j \in [r-1]$ and let $S = A \setminus A_j$. Clearly, the resulting $S$ is uniformly distributed over the $\lceil \frac{(r-2)k}{r-1} \rceil$-element subsets of $A$. Conditional on the choice in the first step, the argument in the previous paragraph implies that the probability that $S$ is an own-subset of $A$ is at least $\frac{1}{r-1}$. As this holds for any outcome of the first step, it also holds unconditionally.

Thus, with each $A \in \FF$ we can associate a family of at least $\frac{1}{r-1} {k \choose \lceil \frac{(r-2)k}{r-1} \rceil}$ own-subsets of $A$ of size $\lceil \frac{(r-2)k}{r-1} \rceil$. The disjoint union of these families over all $A \in \FF$ is contained in ${[n] \choose \lceil \frac{(r-2)k}{r-1} \rceil}$, implying that $|\FF| \cdot \frac{1}{r-1} {k \choose \lceil \frac{(r-2)k}{r-1} \rceil} \le {n \choose \lceil \frac{(r-2)k}{r-1} \rceil}$. It follows that $|\FF| \le (r-1) \frac{{n \choose \lceil \frac{(r-2)k}{r-1} \rceil}}{{k \choose \lceil \frac{(r-2)k}{r-1} \rceil}}$, as claimed.
\end{proof}

\begin{corollary} \label{cor:one-sided}
For $r \ge 3$ and $b = 0,1$ we have
\[ g^{b\dff}_r(n) \le (r-1) \sum_{k=0}^n \frac{{n \choose \lceil \frac{(r-2)k}{r-1} \rceil}}{{k \choose \lceil \frac{(r-2)k}{r-1} \rceil}} = (1 + \frac{r-2}{(r-1)^{\frac{r-1}{r-2}}} + o(1))^n. \]
\end{corollary}

\begin{proof}
As pointed out above, it suffices to treat the case $b=1$. Let $\FF$ be a family of subsets of $[n]$ containing no $1$-focal family of size $r$. Then $|\FF| = \sum_{k=0}^n |\FF \cap {[n] \choose k}|$, and applying the theorem to the families $\FF \cap {[n] \choose k}$ yields the upper bound in summation form.

To obtain the asymptotic expression for the sum, we first use Stirling's formula to approximate ${k \choose \lceil \frac{(r-2)k}{r-1} \rceil}$ up to a factor of order $\sqrt{k}$ by $(\frac{(r-1)^{r-1}}{(r-2)^{r-2}})^{\frac{k}{r-1}}$. Plugging this approximation in the sum gives
\[ \sum_{k=0}^n {n \choose \lceil \frac{(r-2)k}{r-1} \rceil} (\frac{r-2}{(r-1)^{\frac{r-1}{r-2}}})^{\frac{(r-2)k}{r-1}} \]
which, by the binomial formula, is
$\Theta((1 + \frac{r-2}{(r-1)^{\frac{r-1}{r-2}}})^n)$.
\end{proof}

In the case $r=3$, a $1$-focal family is a triple of distinct sets satisfying $A \subseteq B \cup C$, and Corollary~\ref{cor:one-sided}
reproduces the bound of $(\frac{5}{4}+o(1))^n$ obtained by Erd\H{o}s,
Frankl and F\"uredi~\cite{EFF} for the maximum possible cardinality of families not containing such triples. In the case $r=4$, a $1$-focal family is a $4$-tuple of distinct sets satisfying $A \subseteq (B \cup C) \cap (B \cup D) \cap (C \cup D)$, and we get an upper bound of roughly $2^{0.47n}$ for the corresponding extremal problem.

Lower bounds on the extremal functions $g^{b\dff}_r(n,k)$ and $g^{b\dff}_r(n)$ may be obtained, as above, by random choice with alterations. Here, however, one should start with a random subfamily of ${[n] \choose k}$ instead of $\{0,1\}^n$. Optimizing the bounds requires rather messy calculations, which we omit.

\section{Bounds in terms of the set size} \label{sec:k}
We turn our attention now to bounding the cardinality of a family of $k$-element sets (on a ground set of any size) not containing any near-sunflower of size $r$. Note that this question does not make sense for focal families, because we may take arbitrarily many pairwise disjoint $k$-element sets, avoiding focal families of size $3$. With respect to near-sunflowers, however, any upper bound on the cardinality of a $k$-uniform family not containing a sunflower of size $r$ automatically applies to our question, too; in particular, the recent bound of Alweiss, Lovett, Wu and Zhang~\cite{ALWZ}
of order $(\log k)^{(1+o(1))k}$.
Can this be improved for near-sunflowers?

\begin{conjecture} \label{con:k}
Let $r \ge 4$, and let $\FF$ be a family of $k$-element sets which contains no near-sunflower of size $r$. Then $|\FF| \le C^k$, where $C$ is a constant depending only on $r$.
\end{conjecture}

This is a weaker version of the Erd\H{o}s-Rado sunflower conjecture. In view of the fame and difficulty of the latter, this weakening may turn out to be a more accessible goal. But we have not been able to make progress, even for $r=4$.

We do have an upper bound of the desired exponential form under a stronger condition. Saying that a $4$-tuple of distinct sets $A, B, C, D$ is not a near-sunflower can be expressed as follows: there is a way to partition $\{A, B, C, D\}$ into two pairs with intersecting symmetric differences. A natural strengthening is to require this for every pairing of $A, B, C, D$. Let $\FF$ be a family of sets so that for any (ordered) four distinct sets $A, B, C, D \in \FF$ we have $(A \bigtriangleup B) \cap (C \bigtriangleup D) \ne \emptyset$.
K\"orner and Simonyi~\cite{KS} proved that if all sets are subsets of
an $n$-element ground set then $|\FF| \le 1.217^n$ for large $n$.
But here we are interested in such families that are $k$-uniform on any ground set. Fixing $A, B$, the condition implies that any $C$ and $D$ must differ within $A \bigtriangleup B$, which has at most $2k$ elements, easily giving $|\FF| \le 2^{2k}$. The following theorem improves this bound.

\begin{theorem} \label{thm:k}
Let $\FF$ be a family of $k$-element sets so that for any (ordered)
four distinct sets $A, B, C, D \in \FF$ we have $(A \bigtriangleup B) \cap (C \bigtriangleup D) \ne \emptyset$. Then $|\FF| \le 2.148^k$ for large enough $k$.
\end{theorem}

\begin{proof}
Fix two sets $A, B \in \mathcal{F}$ so that $|A \cap B| = t$ maximizes the intersection size over all pairs of distinct sets in $\mathcal{F}$.

For any set $E$, denote by $[E]_t$ a set which is $E$ itself if $|E| \le t$, and otherwise it is an arbitrarily chosen $(t+1)$-element subset of $E$. Let ${A \bigtriangleup B \choose \le t+1}$ be the family of all subsets of $A \bigtriangleup B$ of size at most $t+1$. Define a mapping $f : \mathcal{F} \setminus \{A,B\} \to {A \bigtriangleup B \choose \le t+1}$ by $f(C) = [C \cap (A \bigtriangleup B)]_t$.

We check that $f$ is injective. Let $C$ and $D$ be two distinct sets in $\mathcal{F} \setminus \{A,B\}$. We have to show that $[C \cap (A \bigtriangleup B)]_t \ne [D \cap (A \bigtriangleup B)]_t$. If the sets on both sides have size at most $t$, then $C \cap (A \bigtriangleup B) \ne D \cap (A \bigtriangleup B)$ follows from $(A \bigtriangleup B) \cap (C \bigtriangleup D) \ne \emptyset$. If both $[C \cap (A \bigtriangleup B)]_t$ and $[D \cap (A \bigtriangleup B)]_t$ have size $t+1$, they cannot be equal since that would imply $|C \cap D| \ge t+1$, contradicting the maximality of $t$. Finally, if one of them has size at most $t$ and the other has size $t+1$, they are obviously not equal. This implies that
\begin{equation*}
|\mathcal{F}| - 2 \le \sum_{j=0}^{t+1} {2(k-t) \choose j}.
\end{equation*}

For large $k$, we want to bound the right-hand side from above by $C^k$ for some $C < 2.148$. Let us write $x = \frac{t}{2(k-t)}$. If $x \ge \frac{1}{2}$ then $2(k-t) \le k$ and the sum is bounded by $2^k$. Thus we may assume that $x < \frac{1}{2}$ and approximate the sum by $2^{2(k-t)h(x)} = 2^{\frac{2h(x)}{1+2x}k}$, where $h(x)$ is the binary entropy function. Routine calculations show that the maximum of $2^{\frac{2h(x)}{1+2x}}$ is attained when $x = (1-x)^3$ and its value is less than $2.148$.
\end{proof}

As in all these problems, the probabilistic method can be used to show the existence of a $k$-uniform family $\FF$ with pairwise intersecting symmetric differences, so that $|\FF|$ is exponential in $k$. Our argument gives $|\FF| \approx 1.25^k$, we omit the details.

\bibliographystyle{plain}
\bibliography{near}

\end{document}